\documentclass[11pt]{article}
\usepackage{fullpage}

\usepackage{mathrsfs}
\usepackage{amssymb}
\usepackage{amsmath}
\usepackage{graphicx}
\usepackage{amsmath,amsfonts}
\newtheorem{Theorem}{\quad Theorem}[section]

\newtheorem{Lemma}[Theorem]{\quad Lemma}

\newenvironment{proof}{\noindent\relax{\sc
     Proof}}{{\samepage\par\nopagebreak\hbox
     to\hsize{\hfill$\Box$}}}

\begin{document}
\title{Extinction times for a birth-death process with weak competition}

\author{{\sc Serik Sagitov}\footnote{corresponding author} \ and {\sc Altynay Shaimerdenova} \and {\it Chalmers University of Technology and University of Gothenburg,}\\  {\it and 
Al-Farabi Kazakh National University}} 
\date{}
\maketitle
Published in Lithuanian Mathematical Journal, Vol. 53, No. 2, April, 2013, pp. 220--234

\begin{abstract}
We consider a birth-death process with the birth rates $i\lambda$
and death rates $i\mu +i(i-1)\theta$, where $i$ is the current
state of the process. A positive competition rate
$\theta$ is assumed to be small. In the supercritical case when $\lambda>\mu$
this process can be viewed as a demographic model for a population
with a high carrying capacity around $\lambda-\mu\over\theta$. 

The article reports in a self-contained manner on the asymptotic properties of the
time to extinction for this logistic branching process as $\theta\to0$. All three reproduction regimes $\lambda>\mu$, $\lambda<\mu$, and $\lambda=\mu$ are studied.

\end{abstract}

{\bf Mathematics Subject Classification:} 60J80 \\

{\bf Keywords:} Birth-death process, carrying capacity, time to
extinction, coupling method, logistic
branching process

\section{Introduction}

One of the basic population models with continuous time is the linear birth-death process $(X_{0}(t),t\geq 0)$  with fixed birth and death rates $\lambda$ and $\mu$ per individual. This is a simple example of a branching process describing a population of independently reproducing individuals having three different reproductive regimes: supercritical ($\lambda>\mu$), critical  ($\lambda=\mu$), and subcritical  ($\lambda<\mu$).

The properties of the linear birth-death process $X_0(\cdot)$ and its time to extinction $\tau_0$ are well-known, see for example \cite[pp. 270-2]{GS}. In particular,
\[\mathbb{E}_mX_0(t)=me^{(\lambda-\mu)t}\]
and
$$\mathbb{P}_m(\tau_0\leq t)=\left\{\begin{array}{cl}
 \left(\frac{\mu(1-e^{(\mu-\lambda)t})}{\lambda-\mu e^{(\mu-\lambda)t}}\right)^{m}, &  \mbox{in the sub- and supercritical cases,} \\
(\frac{\lambda t}{1+\lambda t})^{m},  &     \mbox{in the critical case,}
\end{array}
\right.
$$
where $\mathbb{P}_m(\cdot)$ and $\mathbb{E}_m(\cdot)$ stand for the conditional probability and expectation given that the corresponding birth-death process starts from the state $m$. It follows that in the supercritical and critical cases $\mathbb{E}_m(\tau_0)=\infty$ and in the subcritical case  $\mathbb{E}_m(\tau_0)<\infty$.
Letting $t\to\infty$ one obtains the extinction probabilities
\[\mathbb{P}_m(\tau_0<\infty)=\left\{\begin{array}{cl}
 1, &  \mbox{in the subcritical and critical cases,} \\
(\frac{\mu}{\lambda})^{m},  &     \mbox{in the supercritical case.}
\end{array}
\right.
\]
Moreover, it is easy to see that in the subcritical case
\begin{equation}
\mathbb{P}_m\left(\tau_0\le {\ln m+\ln(1-\lambda/\mu)+x\over\mu-\lambda}\right)\to e^{-(e^{-x})},\ \ m\to\infty,
\label{ex1}
\end{equation}
and in the critical case
\begin{equation}
\mathbb{P}_m\left(\tau_0\le mx\right)\to \exp\{-(\lambda x)^{-1}\},\ \ m\to\infty.
\label{ex2}
\end{equation}

The absence of competition among individuals is a major weakness of the linear birth-death population model. 
A natural modification of this simple-minded model is to introduce extra deaths due to competition.
We consider an indexed birth-death process $(X_{\theta}(t),t\geq 0)$ taking non-negative integer values  $i\in\{0,1,2,...\}$ and having time homogenous jump rates
\begin{align}
\left\{\begin{array}{ll}
\mathbb{P}_i(X_\theta(t)=i+1)=\lambda_i t+o(t),& \mbox{with } \lambda_i=i\lambda,\\
\mathbb{P}_i(X_\theta(t)=i-1)=\mu_it+o(t),& \mbox{with } \mu_i=i\mu +i(i-1)\theta,\\
\mathbb{P}_i(X_\theta(t)=i)=1-(\lambda_i +\mu_i)t+o(t),& 
\end{array}
\right.
\label{BD}
\end{align}
as $t\to0$. The key parameters $(\lambda, \mu, \theta)$ of the model are the birth, death, and competition rates providing the following description of the demographic dynamics until the process hits the absorption state $i=0$.

Given the current population size $i\ge1$, the next change in the population size is caused either by a birth or by a death of a particle. It is assumed that coexisting particles give birth independently of each other at rate $\lambda$ per particle, so that interaction among particles does not influence birth events. Particle death is modeled by two parameters: parameter  $\mu$ gives the death rate per particle "due to natural causes" and parameter $\theta$, usually assumed to be small,  quantifies the death rate due to competition pressure (factor $i(i-1)$ appearing in front of $\theta$ represents the number of pairs of competing particles). Putting $\theta=0$ brings us back to the linear  birth-death process $X_0(\cdot)$ mentioned in the Introduction. 

The process $X_\theta(\cdot)$  is an example of the so called logistic branching process studied in \cite{AL} along with its continuous state counterpart. The birth-death framework allows for a more detailed analysis in this special case.
The most conspicuous new feature of $X_\theta(\cdot)$ compared to the linear  birth-death process $X_0(\cdot)$ is the existence of  a threshold
value 
\begin{equation}
i_\theta=\lfloor{\lambda-\mu\over\theta}\rfloor+1
\label{ite}
\end{equation}
in the supercritical case. Obtained from the equation $\lambda_i \approx\mu_i$ the threshold value $i_\theta$ splits the state space in two parts. For
$i<i_\theta$ the process $X_\theta(\cdot)$ tends
to grow while for $i>i_\theta$ it tends to
decrease. A relevant biological interpretation of this threshold
value is the {\it carrying capacity} of the environment for the population in question.

In Section \ref{Sbd} we summarise some useful properties of the time-homogeneous birth-death processes. It follows, in particular, 
that the quadratic form of the death rate compared to the linear birth
rate ensures that our birth-death process with competition goes extinct with probability
one (in contrast with a supercritical linear  birth-death process which never dies out with a positive probability). One of the most interesting characteristics of the process
$X_\theta(\cdot)$ is the random time to extinction $\tau_\theta$.

 If $\theta$ is small, the competition
component $i(i-1)\theta$ is much smaller than $i\mu$ for $i\ll
\theta^{-1}$, so that the process $X_\theta(\cdot)$ at relatively
low levels can be approximated by the linear birth-death process
$X_0(\cdot)$ with parameters $(\lambda, \mu)$ and the same initial
state $X_0(0)=m$. This is done using a coupling construction presented in Section \ref{Scou}.

Section \ref{Smr} presents the main asymptotic results for expected value and distribution of the time to extinction $\tau_\theta$ as $\theta\to0$. The remaining sections contain the proofs.

\section{General properties of time homogeneous birth-death processes}\label{Sbd}

Next we give a short summary of useful results for a time homogeneous birth-death process with birth rates $\lambda_i$ and death rates $\mu_i$, some of these properties can be found in \cite{KM} and \cite{KT}. An important probability 
$$Q_{i}=\mathbb{P}_i(\mbox{reach $i+1$ before } 0)$$ 
satisfies a recursion
\[Q_{i}={\lambda_i\over\lambda_i+\mu_i}+{\mu_i\over\lambda_i+\mu_i}Q_{i-1}Q_i\]
implying
\[{1\over1-Q_{i}}=1+{\lambda_i\over\mu_i}{1\over1-Q_{i-1}}.\]
Using notation
\[\pi_0=1,\  \pi_j={\mu_1\cdots\mu_j\over \lambda_1\cdots\lambda_j},\ \Pi_{0}=0,\ \Pi_k=\sum_{j=0}^{k-1} \pi_j \]
we derive 
$Q_i=\frac{\Pi_{i}}{\Pi_{i+1}}$. 

More generally, for $i\in(k,n)\subset(0,\infty)$
\begin{align*}
\mathbb{P}_i(\mbox{reach $n$  before $k$})&=\frac{\Pi_{i}-\Pi_{k}}{\Pi_{n}-\Pi_{k}},\\
\mathbb{P}_i(\mbox{reach $k$ before $n$})&=\frac{\Pi_{n}-\Pi_{i}}{\Pi_{n}-\Pi_{k}}.
\end{align*}
Using this we can compute the conditional jumping probabilities
\begin{align*}
\mathbb{P}_i(&\mbox{first jump goes down}|\mbox{reach $i+1$ before $k$})\\
&={\mu_i\over\lambda_i+\mu_i}\frac{\mathbb{P}_{i-1}(\mbox{reach $i+1$ before $k$})}{\mathbb{P}_i(\mbox{reach $i+1$ before $k$})}={\mu_i\over\lambda_i+\mu_i}\frac{\Pi_{i-1}-\Pi_{k}}{\Pi_{i}-\Pi_{k}},\end{align*}
which in turn lead to the recursion
\begin{align*}
\beta_i^k&\equiv \mathbb{E}_i(\mbox{time to reach $i+1$}|\mbox{reach $i+1$ before $k$})\\
&={1\over\lambda_i+\mu_i}+{\mu_i\over\lambda_i+\mu_i}\frac{\Pi_{i-1}-\Pi_{k}}{\Pi_{i}-\Pi_{k}}(\beta_{i-1}^k+\beta_i^k)
\end{align*}
resulting in a  difference equation
\begin{align*}
\beta_i^k&={\Pi_{i}-\Pi_{k}\over\lambda_i(\Pi_{i+1}-\Pi_{k})}+{\mu_i(\Pi_{i-1}-\Pi_{k})\over\lambda_i(\Pi_{i+1}-\Pi_{k})}\beta_{i-1}^k,\  \beta_{k+1}^k={1\over\lambda_{k+1}+\mu_{k+1}},
\end{align*}
which is easily solved as
\begin{equation}
\beta_i^k=\frac{\pi_i}{(\Pi_{i}-\Pi_{k})(\Pi_{i+1}-\Pi_{k})}\sum_{j=k+1}^i\frac{(\Pi_{j}-\Pi_{k})^2}{\lambda_j\pi_{j}}.
\label{bi}
\end{equation}

Similarly, for $k\in[1,i]$ the conditional jumping probabilities
\begin{align*}
\mathbb{P}_k(&\mbox{first jump goes up}|\mbox{reach $k-1$ before  $i+1$})\\
&={\lambda_k\over\lambda_k+\mu_k}\frac{\mathbb{P}_{k+1}(\mbox{reach $k-1$ before $i+1$})}{\mathbb{P}_k(\mbox{reach $k-1$ before $i+1$})}={\lambda_k\over\lambda_k+\mu_k}\frac{\Pi_{i+1}-\Pi_{k+1}}{\Pi_{i+1}-\Pi_{k}},\end{align*}
give the recursion
\begin{align*}
\beta_k^{i+1}&\equiv \mathbb{E}_k(\mbox{ time to reach $k-1$}|\mbox{reach $k-1$ before $i+1$})\\
&={1\over\lambda_k+\mu_k}+{\lambda_k\over\lambda_k+\mu_k}\frac{\Pi_{i+1}-\Pi_{k+1}}{\Pi_{i+1}-\Pi_{k}}(\beta_{k+1}^{i+1}+\beta_k^{i+1})
\end{align*}
resulting in a  difference equation
\begin{align*}
\beta_k^{i+1}&=\frac{\Pi_{i+1}-\Pi_{k}}{\mu_k(\Pi_{i+1}-\Pi_{k-1})}+{\lambda_k\over\mu_k}\frac{\Pi_{i+1}-\Pi_{k+1}}{\Pi_{i+1}-\Pi_{k-1}}\beta_{k+1}^{i+1},\ \ \beta_i^{i+1}={1\over\lambda_i+\mu_i},
\end{align*}
which implies
\begin{equation}
\beta_k^{i+1}=\frac{\pi_{k-1}}{(\Pi_{i+1}-\Pi_{k-1})(\Pi_{i+1}-\Pi_{k})}\sum_{j=k}^{i}\frac{(\Pi_{i+1}-\Pi_{j})^2}{\lambda_{j}\pi_{j}}.
\label{bik}
\end{equation}
Observe that for all $v>u\ge0$ relations \eqref{bi} and \eqref{bik} bring
\begin{equation}
\sum_{i=u+1}^v\beta_i^{u}=\sum_{k=u+1}^v\beta_k^{v+1}=\sum_{j=u+1}^{v}\frac{(\Pi_{v+1}-\Pi_{j})(\Pi_{j}-\Pi_{u})}{\lambda_{j}\pi_{j}(\Pi_{v+1}-\Pi_{u})}.
\label{sum}
\end{equation}
This is a confirmation (in terms of the first moments) of  the statement in \cite{Su} claiming that the corresponding conditional hitting times are equal in distribution.

The expected absorption time is given by the formula
\begin{align}
\mathbb{E}_i(\mbox{time to reach 0})&=\sum_{j=0}^{i-1}\pi_j\sum_{k=j+1}^{\infty}\frac{1}{\lambda_k\pi_k}=\sum_{k=1}^{\infty}\frac{\Pi_{k\wedge i}}{\lambda_k\pi_k}. \label{Et}
\end{align}
Indeed, if we denote the last expectation by $\alpha_i$, then the following recursion 
\begin{align*}
\alpha_i=\frac{1}{\lambda_i+\mu_i}+\frac{\lambda_i}{\lambda_i+\mu_i}\alpha_{i+1}+\frac{\mu_i}{\lambda_i+\mu_i}\alpha_{i-1},
\end{align*}
takes place with $\alpha_0=0$. From this recursion it is straightforward  to derive formula \eqref{Et}. It follows from \eqref{Et} that
\begin{equation}
\mathbb{E}_i(\mbox{time to reach $i-1$})=\pi_{i-1}\sum_{k=i}^{\infty}\frac{1}{\lambda_k\pi_k}.
\label{Ett}
\end{equation}

In particular, for the subcritical linear birth-death process formula \eqref{Et} gives
\begin{align*}
\mathbb{E}_m(\tau_0)&={1\over\mu-\lambda}\left(\sum_{k=1}^{m}{1-s^k\over k}+(s^{-m}-1)\sum_{k=m+1}^{\infty}{s^k\over k}\right)\\
&={1\over\mu-\lambda}\left(\sum_{k=1}^{m}k^{-1}+\ln(1-s)+s^{-m}\sum_{k=m+1}^{\infty}{s^k\over k}\right),
\end{align*}
with $s=\lambda/\mu$, implying  
\begin{align}
\mathbb{E}_m(\tau_0)&={\ln m+\gamma+\ln(1-\lambda/\mu)\over\mu-\lambda}+o(1),\ m\to\infty, \label{Et0}
\end{align}
where $\gamma=0,577...$ is Euler's constant. This complements the weak convergence \eqref{ex1} in terms of asymptotic equality of the corresponding expectations.

\section{A coupling to the linear birth-death process}\label{Scou}

To partially extrapolate the nice  properties of the linear birth-death process $X_0(\cdot)$  to the process with interaction $X_\theta(\cdot)$  one can use the following coupling construction (cf \cite{AD}).

Consider a bivariate Markov process $(\widehat{X}_\theta(\cdot),\widehat{X}_0(\cdot))$
with transition rates given in the next list.
\begin{center}
\begin{tabular}{ll}
Type of transition $(0\le i< j)$\ \ \ \ \ \ \ \ \ \ \ \ &Transition rate \\ \\
$(i,i)\to(i+1,i+1)$&$\lambda i$\\
$(i,i)\to(i-1,i-1)$&$\mu i $\\
$(i,i)\to(i-1,i)$&$\theta i(i-1)$\\
$(i,j)\to(i+1,j)$&$\lambda i$\\
$(i,j)\to(i-1,j)$&$\mu i+\theta i(i-1)$\\
$(i,j)\to(i,j+1)$&$\lambda j $\\
$(i,j)\to(i,j-1)$&$\mu j $
\end{tabular}
\end{center}
The process is constructed in such a way that $\widehat{X}_\theta(t) \leq \widehat{X}_0(t)$
for all $t\geq 0$, and the
marginal distributions of
$(\widehat{X}_\theta(\cdot),\widehat{X}_0(\cdot))$ coincide
with those of $X_\theta(\cdot)$ and $X_0(\cdot)$, respectively.

An important question here is how
long this bivariate process stays at the diagonal if $(\widehat{X}_\theta(0),\widehat{X}_0(0))=(m,m)$.  Let $\kappa_\theta$ be the number of jumps  of the process $(\widehat{X}_\theta(\cdot),\widehat{X}_0(\cdot))$ until separation, if the components stay together until extinction we put $\kappa_\theta=\infty$. We show below that
\begin{equation}
\mathbb{P}_{(m,m)}(\kappa_\theta\le n)\le\frac{(m+n)n\theta}{\lambda+\mu},
\label{ka}
\end{equation}
where $\mathbb{P}_{(m,m)}(\cdot)$ stands for the probability conditioned on the bivariate process starting from the state $(m,m)$.

Suppose $\theta\to0$ and the starting level $m$ is fixed.
In the subcritical and critical cases the total number of births and deaths in the linear birth-death processes is almost surely finite and due to \eqref{ka} we may conclude that $\tau_\theta\rightarrow\tau_0$ almost surely. Moreover, since a supercritical branching process conditioned on extinction behaves like a subcritical branching process, we obtain that $\tau_\theta\rightarrow\tau_0$ almost surely provided $\tau_0<\infty$. This observation is summarised in the next section as a part of Theorem \ref{Tsub}.

To prove \eqref{ka} observe that
$$\kappa_\theta=\inf\{k:U_k\neq V_k\},$$ where $(U_k,V_k)_{k\ge0}$ are the consecutive states visited by of the process
$(\widehat{X}_\theta(\cdot),\widehat{X}_0(\cdot))$. Note that the only way for the bivariate process to
get off the diagonal is the move $(i,i)\to(i-1,i)$ having the
probability $\theta (i-1)\over \lambda+\mu+\theta (i-1)$ which is
negligible, if the current level $i$ is not too high. Since
\begin{align*}
\mathbb{P}_{(m,m)}(\kappa_\theta=n|\kappa_\theta>n-1)&=\sum_{i=1}^{m+n-1} \frac{\theta (i-1)}{\lambda +\mu + \theta (i-1)}\mathbb{P}_{(m,m)}(U_{n-1}=i)\\
&\leq \frac{\theta}{\lambda+\mu} \mathbb{E}_{(m,m)}U_{n-1}\leq \frac{(m+n)\theta}{\lambda+\mu},
\end{align*}
we derive \eqref{ka} 
\begin{align*}
\mathbb{P}_{(m,m)}(\kappa_\theta\le n)&=1-\prod_{k=1}^n \mathbb{P}_{(m,m)}(\kappa_\theta>k|\kappa_\theta>k-1)\\
&\leq\sum_{k=1}^n\frac{(m+k)\theta}{\lambda+\mu}\le\frac{(m+n)n\theta}{\lambda+\mu}.
\end{align*}

\section{Main Results}\label{Smr}

We claim that as $\theta\to0$  the following two limit theorems hold for the birth-death process defined by \eqref{BD}.

\begin{Theorem} \label{Tsub}
If $X_\theta(0)=m$, where $m$ is a fixed positive integer, then

(i) in the subcritical and critical cases  when $\lambda\le\mu$
$$\mathbb{P}_{(m,m)}(\tau_\theta\rightarrow\tau_0)=1,$$

(ii) in the supercritical case when $\lambda>\mu$
$$\mathbb{P}_{(m,m)}(\tau_\theta\rightarrow\tau_0|\tau_0<\infty)=1,$$
and for any $x\ge0$
    $$\mathbb{P}_{(m,m)}(\tau_\theta>xc_1\sqrt\theta\, e^{c_2/\theta})|\tau_0=\infty)\rightarrow e^{-x},$$
    where 
    \begin{equation}
c_1=\lambda(\lambda-\mu)^{-2}\sqrt{2\pi/\mu},Ê\hspace{5mm}c_2=\lambda-\mu-\mu\ln(\lambda/\mu).
\label{c}
\end{equation}
    \end{Theorem}

Theorem \ref{Tsub} (i) and the first part of (ii) are proven in
the previous section. The proof of the second part of Theorem
\ref{Tsub} (ii) is given after the proof of the first part of the
next theorem.

\begin{Theorem}\label{Tsup} If $X_\theta(0)=m_\theta$ and $\theta m_\theta\rightarrow a>0$, then

(i)  in the supercritical case when $\lambda>\mu$  $$\mathbb{E}_{m_\theta}(\tau_\theta)\sim c_1\sqrt\theta\, e^{c_2/\theta}$$ 
with positive constants $c_1$, $c_2$ given by \eqref{c}, and for any $x\ge0$
    $$\mathbb{P}_{m_\theta}(\tau_\theta>xc_1\sqrt\theta\, e^{c_2/\theta})\rightarrow e^{-x},$$

(ii) in the subcritical case when $\lambda<\mu$  
\begin{equation}
\mathbb{E}_{m_\theta}(\tau_\theta)={\ln(a\theta^{-1})+\ln{\mu-\lambda\over \mu}+\ln{\mu-\lambda\over\mu-\lambda+a}+\gamma\over\mu-\lambda}+o(1),
\label{sme}
\end{equation}
and for any $x\ge0$
\begin{equation}
\mathbb{P}_{m_\theta}\left(\tau_{\theta}\le {\ln(a\theta^{-1})+\ln{\mu-\lambda\over \mu}+\ln{\mu-\lambda\over\mu-\lambda+a}+x\over\mu-\lambda}\right)\rightarrow e^{(-e^{-x})}.
\label{2ii}
\end{equation}

(iii) in the critical case when $\lambda=\mu$
$$\mathbb{E}_{m_\theta}(\tau_\theta)\sim{(\pi/2)^{3/2}\over\sqrt{\theta \mu}}.$$
\end{Theorem}

The asymptotic formulae for $\mathbb{E}_{m_\theta}(\tau_\theta)$ in Theorem  \ref{Tsup} (ii), (iii) are verified by simulations as shown in Figures \ref{fig1}, \ref{fig2}. Comparing the asymptotic formula  \eqref{Et0} for the linear birth-death process to the that for the process with competition \eqref{sme} we see that as $\theta\to0$ and $\theta m_\theta\to a$ the average survival time reduces by
\[\mathbb{E}_{m_\theta}(\tau_0)-\mathbb{E}_{m_\theta}(\tau_\theta)\to{1\over\mu-\lambda}\ln{\mu-\lambda+a\over\mu-\lambda}.\]
As one would expect, this difference becomes small for larger values of $\mu$ and/or smaller values of $a$.

\begin{figure}
\begin{center}
\includegraphics[trim=0mm 60mm 0mm 40mm, clip, height=8cm]{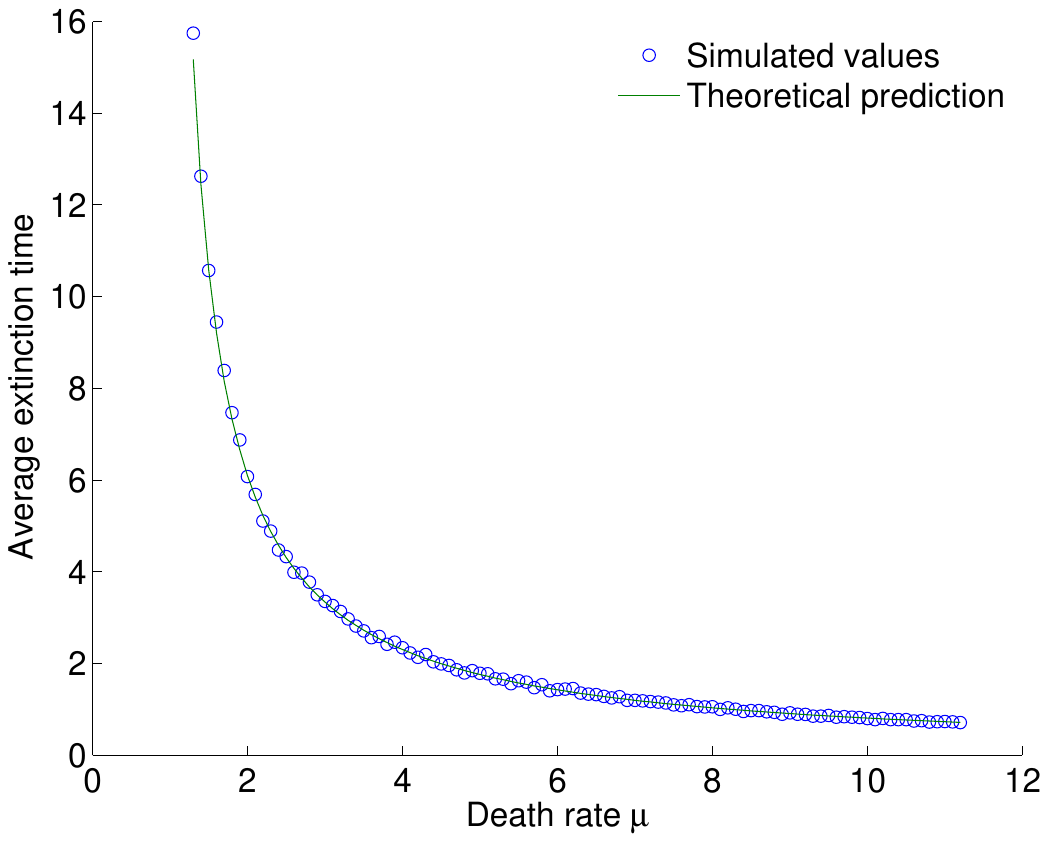} 
\end{center}
\caption{Averages of 100 simulations for each value of the death rate are plotted against the values predicted by Theorem  \ref{Tsup} (ii). 
Choice of parameters: initial population size $m_\theta=1000$, competition strength $\theta=0.001$, and birth rate $\lambda=1$. 
}
\label{fig1}
\end{figure}

Notice that Theorem  \ref{Tsup} (ii) provides with a counterpart of the weak convergence \eqref{ex1} for the linear birth-death processes, however, we could not find a counterpart of \eqref{ex2} in the critical case. The following lemma plays a crucial role in the asymptotic analysis of all three cases.
\begin{Lemma} \label{lem}
For our particular model the function $\pi_j=\prod_{i=0}^{j-1}{\mu+\theta i\over \lambda}$ satisfies the approximation
 \begin{equation}
\pi_{j}=(1+j\theta/\mu)^{-1/2}\, e^ {-W(j\theta)/\theta}(1+\eta_j(\theta)),\ j\ge 1,
\label{pij}
\end{equation}
where $W(x)= x-x \ln \frac{\mu + x}{\lambda} - \mu \ln  \frac{\mu + x}{\mu}$ and for any fixed $T>0$
$$\sup_{1\le j\le T/\theta}\left|\eta_j(\theta)\right|\to0,\ \ \theta\to0.$$
\end{Lemma}
\begin{figure}
\begin{center}
 \includegraphics[trim=0mm 60mm 0mm 40mm, clip, height=8cm]{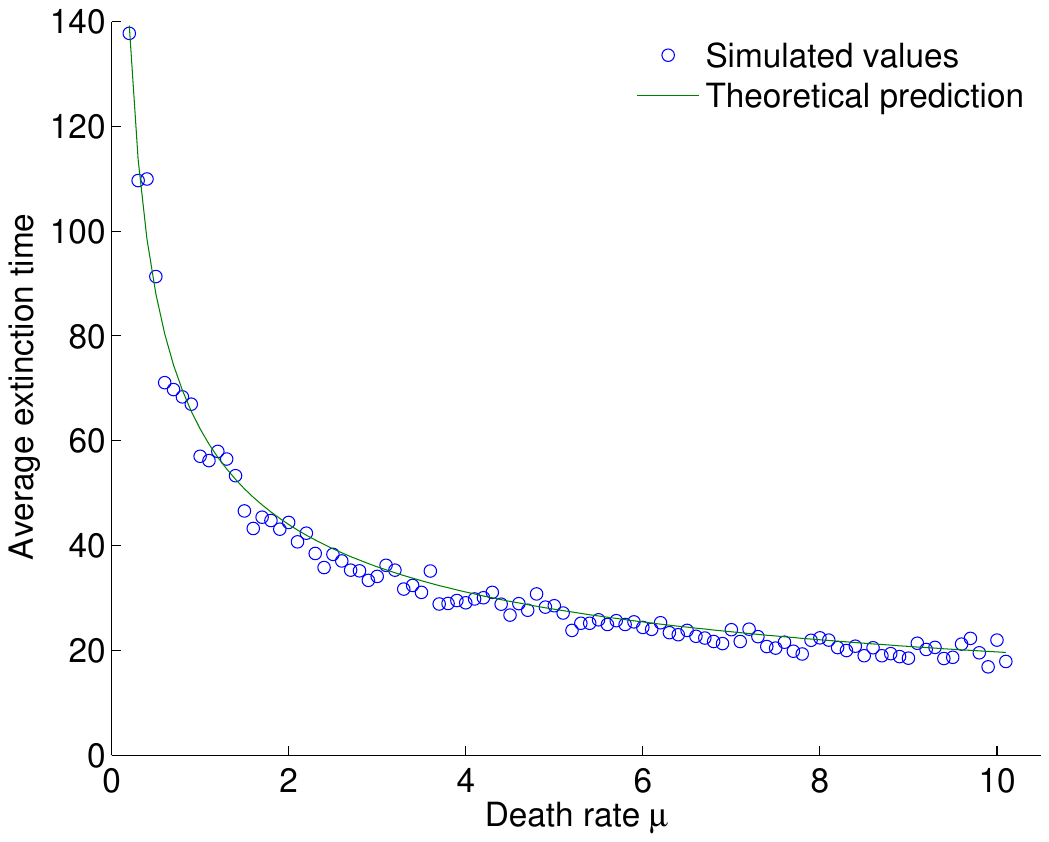}
\end{center}
\caption{Averages of 100 simulations for $m_\theta=1000$, $\theta=0.001$ are plotted against the values predicted by Theorem  \ref{Tsup} (iii). 
}
\label{fig2}
\end{figure}

\begin{proof} 
Observe that
\begin{align*}
\theta\ln \pi_{j}&=\theta\sum\limits_{k=0}^{j-1}\ln\frac{\mu+k\theta}{\lambda}\\
&= \int_0^{j\theta} \! \ln\frac{\mu +y}{\lambda} \, \mathrm{d}
y -\sum_{k=0}^{j-1} \int_{k\theta}^{(k+1)\theta} \ln\frac{\mu +
y}{\mu + k\theta}\,dy\\
&= -W(j\theta) - \frac{\theta}{2}\sum_{k=0}^{j-1}
\frac{\theta}{\mu +k\theta}+O(j\theta^3).
\end{align*}
It remains to verify that 
$$\sup_{j\ge0}\left|\prod_{k=0}^je^{-{\theta\over\mu+k\theta}}-{\mu\over\mu+j\theta}\right|\to0.$$
\end{proof}

\section{Proofs for the supercritical case}

In this section we first prove Theorem \ref{Tsup} (i) and then Theorem \ref{Tsub} (ii) borrowing key ideas from \cite{AD}.

 We start by considering the supercritical $X_\theta(\cdot)$ with the initial state $X_\theta(0)=i_\theta$ given by  \eqref{ite}. It will take a geometric number $K_\theta\sim{\rm Geom}(1-Q_{i_\theta})$ of returns to the initial state from above before the extinction event. Let $\tau_{\theta}'$ be the time needed for $X_\theta(\cdot)$ to enter the level $i_\theta$ from above, and $\tau_{\theta}''$  be the
    absorption time counted from the last entrance moment to the state $i_\theta$ from above. If there were no visits of $i_\theta$ from above, we put $\tau_{\theta}'=\tau_{\theta}$  and $\tau_{\theta}''=0$.
Clearly,  $\tau_{\theta}$ is the sum of  $\tau_{\theta}'$,  $\tau_{\theta}''$, and of $K_\theta$ independent durations of the corresponding excursions.  It follows that the statement (ii) of Theorem \ref{Tsup} is a straightforward consequence of the next three lemmata.

\begin{Lemma} \label{l1}
In the supercritical case as $\theta\to0$
\begin{equation}
1-Q_{i_\theta}\sim
 \frac{(\lambda-\mu)\sqrt{\mu}}{\lambda^{3/2}} e^ {-c_2/\theta}.
\label{qi}
\end{equation}
\end{Lemma}
\begin{Lemma} \label{l2}
In the supercritical case  the expected duration $M_\theta$ of an excursion starting from $i_\theta$ and returning to $i_\theta$ from above satisfies
    $$ M_{\theta}\sim \sqrt{\frac{2\pi\theta}{\lambda}}\frac{1}{\lambda-\mu},\ \theta\to0.$$
\end{Lemma}
\begin{Lemma} \label{l3}
Under the assumptions of Theorem \ref{Tsup} (i) for any fixed positive $\epsilon$
    $$\mathbb{E}_{m_\theta}(\tau_{\theta}')=o(e^{\epsilon/\theta}),\ \ \mathbb{E}_{m_\theta}(\tau_{\theta}'')=o(e^{\epsilon/\theta}),\ \ \theta\to0.$$
\end{Lemma}

\begin{proof} {\sc of Lemma \ref{l1}}.
It is shown in Section \ref{Sbd} that $1-Q_{i_\theta}=\frac{\pi_{i_\theta}}{\Pi_{i_\theta+1}}$.
According to Lemma \ref{lem}
\begin{equation}
\pi_{i_\theta}\sim\sqrt{\mu\over\lambda}\,e^{-c_2/\theta},
\label{pit}
\end{equation}
and in view of $\Pi_{i_\theta}\to\frac{\lambda}{\lambda-\mu}$ we arrive at \eqref{qi}.
\end{proof}

\begin{proof} {\sc of Lemma \ref{l2}}.
We show first that
\begin{align} 
\sum_{j=1}^\infty\frac{1}{\lambda_j \pi_j}\sim \frac{1}{(\lambda-\mu)
\sqrt{\mu}}\sqrt{2\pi\theta}e^{c_2/\theta}
\label{lap}
\end{align}
by dividing the left sum in two parts using the threshold $l_\theta=\lceil{2\lambda-\mu\over\theta}\rceil$.  Observe that by definitions of $i_\theta$ and
$l_\theta$ we have 
$\pi_j\geq
\pi_{i_\theta}\cdot 2^{j-l_\theta}$ for
$j>l_\theta$. Thus applying \eqref{pit} we obtain
$$\sum_{j=l_\theta+1}^{\infty} \frac{1}{\lambda_j \pi_j}=O(\theta/\pi_{i_\theta})=O(\theta e^{c_2/\theta}).$$

On the other hand, due to \eqref{pij}
$$\sum_{j=1}^{l_\theta} \frac{1}{j\pi_j}
\sim\int_0^{2\lambda-\mu}\frac{\sqrt{\mu+s}}{s\sqrt{\mu}}e^{W(s)/\theta}ds.$$
As the function $W(\cdot)$ has its maximum over the integration interval at the point $\lambda-\mu$ we conclude using the Laplace method that
$$\sum_{j=1}^{l_\theta} \frac{1}{\lambda_j\pi_j}\sim\frac{\sqrt{2\pi\theta}e^{c_2/\theta}}{(\lambda-\mu)\sqrt{\mu}}$$ 
completing the proof of \eqref{lap}.

Combining \eqref{bi} and \eqref{Ett} we get
$$ M_{\theta}=\pi_{i_\theta}\sum_{k=1}^\infty\frac{\psi_{k}(i_\theta)}{\lambda_k\pi_{k}}, \mbox{ where }\psi_{k}(i)=\min\left\{\frac{\Pi_{k}^2}{\Pi_{i}\Pi_{i+1}},1\right\}$$
Relations \eqref{pit} and   \eqref{lap} give
$$\pi_{i_\theta}\sum_{k=1}^\infty\frac{1}{\lambda_k\pi_{k}}\sim \sqrt{\frac{2\pi\theta}{\lambda}}\frac{1}{\lambda-\mu},$$
and it remains only to observe that $\psi_{k}(i_\theta)\to1$ uniformly over $k$ larger than $\epsilon/\theta$ however small  is a fixed positive $\epsilon$.
\end{proof}

\begin{proof} {\sc of Lemma \ref{l3}}.
Notice that due to \eqref{sum}
$$E(\tau_{\theta}'')<\mathbb{E}_{i_\theta}(\mbox{time to reach 0 } | \mbox{ reach 0 before $i_\theta+1$})=\sum_{k=1}^{i_\theta}\beta_k^{i_\theta+1}=o(e^{\epsilon/\theta}).$$
On the other hand, for $m_\theta\le i_\theta$
\begin{align*}
\mathbb{E}_{m_\theta}(\tau_{\theta}')&<\mathbb{E}_{1}(\mbox{time to reach $i_\theta+1$ } | \mbox{ reach $i_\theta+1$ before 0})\\
& \hspace{5mm}+\mathbb{E}_{i_\theta+1}(\mbox{time to reach $i_\theta$})\\
&=\sum_{i=1}^{i_\theta}\beta_i^0+\pi_{i_\theta}\sum_{k=i_\theta+1}^{\infty}\frac{1}{\lambda_k\pi_k}=o(e^{\epsilon/\theta}),
\end{align*}
and for $m_\theta> i_\theta$
\begin{align*}
\mathbb{E}_{m_\theta}(\tau_{\theta}')&=\mathbb{E}_{m_\theta}(\mbox{time to reach $i_\theta$})=\sum_{k=i_\theta+1}^{\infty}\frac{\Pi_{k\wedge m_\theta}-\Pi_{i_\theta}}{\lambda_k\pi_k}=o(e^{\epsilon/\theta}).
\end{align*}

\end{proof}

\begin{proof} {\sc of Theorem \ref{Tsub}} (ii).
Put 
$$S_{\theta}(\delta)=\inf\{t\geq 0: \widehat{X}_0(t)
\geq \delta/\theta\}.$$
In view of the previous analysis it is enough to show that for some fixed $\delta\in(0,{\lambda-\mu\over\mu})$
$$\mathbb{P}_{(m,m)}(\tau_{\theta}> S_{\theta}(\delta)|\tau_0 = \infty) \rightarrow 1.$$
We verify this by showing that for some fixed $\alpha\in(0,\frac{1}{2})$, $\rho\in(0,\frac{\lambda-\mu}{\lambda+\mu})$
\begin{equation}
\mathbb{P}_{(m,m)}(\tau_{\theta}> S_{\theta}(\rho\theta^{1-\alpha})|\tau_0 = \infty) \rightarrow 1,
\label{al}
\end{equation}
and
\begin{equation}
\mathbb{P}_{m}(\tau_{\theta}> S_{\theta}(\delta)|\tau_{\theta}> S_{\theta}(\rho\theta^{1-\alpha})) \rightarrow 1.
\label{ad}
\end{equation}

According to \eqref{ka} we have
\[\mathbb{P}_{(m,m)}(U_k=V_k, k=0,\ldots,\theta^{-\alpha}|\tau_0 = \infty) \rightarrow 1.
\]
Note that 
$$ \mathbb{P}_m(V_{\theta^{-\alpha}}\geq \rho\theta^{-\alpha}|V_{k}\neq 0; 1\leq
k\leq \theta^{-\alpha})\rightarrow 1,$$ 
since under the condition of
non-extinction $V_n$ is just a simple random walk restricted to
the set of positive integers, having a drift that is bounded from
below by $\frac{\lambda-\mu}{\lambda+\mu}>0$. Combining the last two relations we arrive at \eqref{al}.

Finally, \eqref{ad} follows from the fact that the probability
$$\mathbb{P}_{i}(\mbox{reach $n$  before $0$})={\Pi_{i}}/{\Pi_{n}}$$
with $i=\rho\theta^{-\alpha}$ and $n=\delta/\theta$ tends to one as $\theta\to0$.
\end{proof}

\section{Proof of Theorem \ref{Tsup} (ii)}
In the subcritical case $s=\lambda/\mu$ lies in $(0,1)$. To establish the approximation formula   \eqref{sme} we refer to \eqref{Et} which gives
\begin{align*}
\mathbb{E}_m(\tau_\theta)&=\sum_{j=0}^{m-1}\sum_{k=j+1}^{\infty}\frac{\pi_j}{\lambda k\pi_k}=\lambda^{-1}\sum_{j=0}^{m-1}\sum_{k=j+1}^{\infty}\frac{s^{k-j}}{ k}\cdot
r_j\cdots r_{k-1},
\end{align*}
where $r_i={\mu\over\mu+i\theta}$, and on the other hand,
\begin{align*}
\mathbb{E}_m(\tau_0)&=\lambda^{-1}\sum_{j=0}^{m-1}\sum_{k=j+1}^{\infty}\frac{s^{k-j}}{ k}.
\end{align*}
Thus in view of  \eqref{Et0} we have to verify that
\begin{equation}
\sum_{j=0}^{a/\theta}\sum_{k=j+1}^{\infty}\frac{s^{k-j}}{ k}\cdot\big(1-
r_j\cdots r_{k-1}\big)
\to{\lambda\over\mu-\lambda}\ln{\mu-\lambda+a\over\mu-\lambda}.\label{hope} 
\end{equation}

To prove \eqref{hope} choose arbitrary but fixed small $\epsilon$ and large $T$ and consider
\begin{align*}
\sum_{j=\epsilon/\theta}^{a/\theta}\sum_{k=j+1}^{j+T}\frac{s^{k-j}}{ k}\cdot\big(1-r_j\cdots r_{k-1}\big)
&=\sum_{j=\epsilon/\theta}^{a/\theta}\sum_{k=j+1}^{j+T}\frac{s^{k-j}}{ k}\cdot\big(1-e^{-{V(k\theta)-V(j\theta)\over\theta}}
\big)+o(1),
\end{align*}
where a counterpart of Lemma \ref{lem} was used with 
\begin{align*}
V(x)&=(x+\mu)\ln{x+\mu\over\mu}-x,\\
V(x)-V(y)&=(x-y)\ln{x+\mu\over\mu}+(y+\mu)\Big(\ln{x+\mu\over y+\mu}-{x-y\over y+\mu}\Big).
\end{align*}
Due to the last equality we can replace $e^{-{V(k\theta)-V(j\theta)\over\theta}}$ with $\big({\mu\over j\theta+\mu}\big)^{k-j}$ and get
\begin{align*}
\sum_{j=\epsilon/\theta}^{a/\theta}\sum_{k=j+1}^{j+T}\frac{s^{k-j}}{ k}\cdot\big(1-r_j\cdots r_{k-1}\big)
&=\sum_{j=\epsilon/\theta}^{a/\theta}\sum_{l=1}^{T}\frac{s^{l}}{j}\cdot\big(1-(1+j\theta/\mu)^{-l}\big)+o(1).
\end{align*}
Since 
\[\sum_{l=1}^{\infty}s^{l}\big(1-(1+j\theta/\mu)^{-l}\big)={1\over1-s}-{1\over1-s(1+j\theta/\mu)^{-1}}={s\over1-s}\cdot{j\theta\over\mu(1-s)+j\theta},\]
to derive \eqref{hope} it remains to observe that
\begin{align*}
\sum_{j=\epsilon/\theta}^{a/\theta}\sum_{l=1}^{T}\frac{s^{l}}{j}\cdot\big(1-(1+j\theta/\mu)^{-l}\big)={s\over1-s}\sum_{j=\epsilon/\theta}^{a/\theta}{\theta\over\mu(1-s)+j\theta}+
\rho_T(\theta),
\end{align*}
where 
$$\limsup_{T\to\infty}\limsup_{\theta\to\infty}|\rho_T(\theta)|=0,$$
and
\begin{align*}
\sum_{j=\epsilon/\theta}^{a/\theta}{\theta\over\mu(1-s)+j\theta}&\to\int_\epsilon^a\frac{dx}{\mu(1-s)+x}=\ln{\mu-\lambda+a\over\mu-\lambda+\epsilon}.
\end{align*}
This finishes the proof of  \eqref{sme}.

Next we prove the weak convergence stated  in the subcritical case. Fix some  $0<\alpha<\frac{1}{2}$. Following the approach of \cite{B75}, we establish  \eqref{2ii} after splitting the extinction time $\tau_{\theta}$ in two parts
\[\tau_{\theta}=\tau_{\theta,1}+\tau_{\theta,2},\]
where  $\tau_{\theta,1}$ is the time for
$X_{\theta}(\cdot)$ to reach the level $\theta^{-\alpha}$ and $\tau_{\theta,2}$ is the time for the process $X_{\theta}(\cdot)$ starting from $\theta^{-\alpha}$ to get absorbed at 0.

If $X_\theta(0)=m_\theta$ and $\theta m_\theta\rightarrow a>0$, then according to \cite{Ku} the
scaled process $\theta X_{\theta}(\cdot)$ converges in
probability, uniformly on compact time intervals, to the
deterministic motion $x(\cdot)$
governed by the differential equation
\begin{equation}
x'(t)=(\lambda-\mu)x(t)-x^{2}(t),\ \ x(0)=a.
\label{eq}
\end{equation}
This equation has an explicit solution
\begin{equation}
{1\over x(t)}=\left({1\over a}+{1\over\mu-\lambda}\right)e^{(\mu-\lambda)t}-{1\over\mu-\lambda}.
\label{eq1}
\end{equation}
Solving formally $x(t)=\theta^{1-\alpha}$ for the time $t$ required for the deterministic motion to reach the low level $\theta^{1-\alpha}$ we find
\begin{equation}
\tau_{\theta,1}=\frac{(1-\alpha)\ln\theta^{-1}-\ln (a^{-1}+(\mu-\lambda)^{-1})}{\mu-\lambda}+o(1)
\label{tau1}
\end{equation}
in probability. Combining \eqref{ex1} with \eqref{ka} entails
\[P\left(\tau_{\theta,2}\le {\alpha\ln \theta^{-1}+\ln(1-\lambda/\mu)+x\over\mu-\lambda}\right)\to e^{-(e^{-x})},\ \ \theta\to0,\]
which together with \eqref{tau1} give \eqref{2ii}.

The full justification of  \eqref{tau1} can be achieved using the approach developed in \cite{B74} and \cite{B75}. It is based on an appropriate integral of the equation \eqref{eq}, which in our case is
\begin{equation}
h(z,t)=t-{\ln(\mu-\lambda+z)-\ln x+\ln a-\ln(\mu-\lambda+a)\over \mu-\lambda}.
\label{hxt}
\end{equation}
If $x(t)$ satisfies  \eqref{eq1}, then $h(x(t),t)=0$ and furthermore, $x(t-h(z,t))=z$.
It follows,
\begin{equation}
|z-x(t)|\le (\mu-\lambda+a)x(t)|h(z,t)|.
\label{ine}
\end{equation}

For the rest of the proof we replace $a$ by $\theta m_\theta$ in relations \eqref{eq1} and \eqref{hxt} defining $x(t)$ and $h(x,t)$. Let $\nu_\epsilon$ denote the minimal $t>0$ such that $|\theta X_\theta(t)-x(t)|>\epsilon$, and put $H_\theta(t)=|h(\theta X_\theta(t),t)|$ so that $H_\theta(0)=0$. According to  \cite{B75} a modified Corollary 1 of Lemma 5 in \cite{B74} gives
\[\mathbb{P}_{m_\theta}(H_\theta(t\wedge\nu_\epsilon)>y)\le 2\exp\{-ky+tC_\epsilon(\theta,k,t)\}\]
for all positive $t$ and $k$, where the function $C_\epsilon(\theta,k,t)$ can be chosen such that for some positive constants $C_1,C_2,C_3$
\[C(\theta,k,t)={C_1k\theta \over (x(t)-\epsilon)^{2}}+{C_2k^2\theta \over (x(t)-\epsilon)^{2}}\exp\left\{{C_3k\theta\over x(t)-\epsilon}\right\},\]
if we assume that $x(t)>\epsilon$.
If furthermore, $x(t)-\epsilon>C_4\theta^{1-\alpha}$, then
\[C(\theta,k,t)<C_5k\theta^{2\alpha-1}+C_6k^2\theta^{2\alpha-1}e^{C_7k\theta^{\alpha}}.\]

\section{Proof of Theorem \ref{Tsup} (iii)}

According to \eqref{Et} and \eqref{pij} we have in the critical case
\begin{align*}
\theta \mathbb{E}_{m_\theta}(\tau_{\theta})&\sim\int_0^{a}{\sqrt{\mu+y}\over\mu y}\int_0^{y}{1\over\sqrt{\mu+x}}e^{W(y)-W(x)\over\theta}dxdy\\
&\hspace{10mm}+\int_{a}^\infty{\sqrt{\mu+y}\over\mu y}e^{W(y)/\theta}dy \cdot \int_0^{{a}}{1\over\sqrt{\mu+x}}e^{-W(x)/\theta}dx,
\end{align*}
where $W(x)= x-(\mu+x) \ln  \frac{\mu + x}{\mu}$.
Notice that $W(x)=-{x^2\over2\mu}(1+2r(x))$, where $r(x)\to0$ as $x\to0$. It follows, that for any $T>0$
\begin{align*}
{1\over\sqrt{\mu\theta}}\int_0^{T\sqrt{\mu\theta}}&{\sqrt{\mu+y}\over \mu y}\int_0^{y}{1\over\sqrt{\mu+x}}e^{W(y)-W(x)\over\theta}dxdy\\
&=\int_0^{T}{\sqrt{\mu+z\sqrt{\mu\theta}}\over \mu z}\int_0^{z}{1\over\sqrt{\mu+t\sqrt{\mu\theta}}}e^{t^2-z^2\over2}e^{t^2r(t\sqrt{\mu\theta})}e^{-z^2r(z\sqrt{\mu\theta})}dtdz\\
&\to\mu^{-1}\int_0^{T}z^{-1}\int_0^{z}e^{t^2-z^2\over2}dtdz.
\end{align*}
On the other hand, since for $0\le x\le y$
$$W(y)-W(x)\le(x-y)\ln  \frac{\mu + x}{\mu}-{(y-x)^2\over2(\mu+y)},$$
we have with $C={\sqrt{\mu+a}\over\sqrt{\mu}}$
\begin{align*}
\int_{T\sqrt{\mu\theta}}^{a}&{\sqrt{\mu+y}\over y}\int_0^{y}{1\over\sqrt{\mu+x}}e^{W(y)-W(x)\over\theta}dxdy\\
&\le C\int_{T\sqrt{\mu\theta}}^{a}y^{-1}\left(\int_0^{y/2} e^{-{(y-x)^2\over2(\mu+a)\theta}}dx+\int_{y/2}^{y} e^{{x-y\over\theta}\ln  \frac{\mu + x}{\mu}}dx\right)dy\\
&\le {C\over2}\int_{T\sqrt{\mu\theta}}^{\infty} e^{-{y^2\over8(\mu+a)\theta}}dy+C\int_{T\sqrt{\mu\theta}}^{a}y^{-1}\int_{y/2}^{y} e^{{x-y\over\theta}\ln  \frac{\mu + y/2}{\mu}}dxdy\\
&\le {C\sqrt{\theta}\over2}\int_{T\sqrt{\mu}}^{\infty} e^{-{z^2\over8(\mu+a)}}dz+C\theta\int_{T\sqrt{\mu\theta}}^{a} {dy\over y\ln (1 +{ y\over2\mu})},
\end{align*}
where the last integral is estimated from above by a constant plus
\begin{align*}
\int_{T\sqrt{\mu\theta}}^{2\mu} {dy\over y\ln (1 +{ y\over2\mu})}&\le\int_{T\sqrt{\theta/(4\mu)}}^{1} {dz\over z(z-{z^2\over2})}\le{2\sqrt{\mu}\over T\sqrt{\theta}}+{1\over2}\ln\left({4\mu\over T\sqrt{\mu\theta}}\right).
\end{align*}
Using a table integral
\begin{align*}
\int_0^{\infty}z^{-1}\int_0^{z}e^{t^2-z^2\over2}dtdz=\left({\pi\over2}\right)^{3/2}
\end{align*}
we conclude that
\begin{align*}
\limsup_{\theta\to0}\left|{1\over\sqrt{\theta}}\int_0^{a}{\sqrt{\mu+y}\over\mu y}\int_0^{y}{1\over\sqrt{\mu+x}}e^{W(y)-W(x)\over\theta}dxdy-{1\over\sqrt{\mu}}\left({\pi\over2}\right)^{3/2}\right|\le \epsilon_T,
\end{align*}
where $\epsilon_T\to0$ as $T\to\infty$.

It remains to observe that
\begin{align*}
\int_{a}^\infty{\sqrt{\mu+y}\over\mu y}e^{W(y)/\theta}dy \cdot \int_0^{{a}}{1\over\sqrt{\mu+x}}e^{-W(x)/\theta}dx=o(\sqrt{\theta}).
\end{align*}

\noindent{\sc Remark}. Our approximations for the mean extinction time are specific to the population model we study. These should be compared with similar calculations performed in a more general setting by \cite{Sar}, where, however, strict justifications of some important steps are missing.

\vspace{5mm}

\noindent\textbf{Acknowledgments}. SS was supported by the Swedish Research Council grant 621-2010-5623. AS was supported by the Scientific Committee of  Kazakhstan's Ministry of Education and Science, grant 0732/GF 2012-14.


\begin{thebibliography}{99}

\bibitem{AD}{Andersson, H., and Djehiche, B. (1998). A threshold limit theorem for the stochastic logistic epidemic. J. Appl. Prob., 35(3) : 662-670.}

\bibitem{B74}{Barbour, A.D. (1974). On a functional central limit theorem for Markov population processes. Adv. Appl. Prob., 6(1): 21-39.}
\bibitem{B75}{Barbour, A.D. (1975). The duration of closed stochastic epidemic. Biometrika, 62(2): 477-482.}
\bibitem{Sar}{Doering, C.R., Sargsyan, K.V., and Sander, L.M. (2005). Extinction times for birth-death processes:exact results,continuum asymptotics, and the failure of the Fokker-Plank approximation. Multiscale Model. Simul.,
3(2): 283-299.}
\bibitem{GS}{ Grimmet, G.R. and  Stirzaker, D.R. (2001). Probability and Random Processes (3rd Edition). Oxford: Clarendon Press.}
\bibitem{KM} {Karlin, S., and McGregor, J. (1957). The classification of birth and death processes. Trans. Amer. Math. Soc., 86(2): 366-400.}
\bibitem{KT} {Karlin, S. and Taylor, M. (1975). A first course in stochastic processes (2nd Edition). New York: Academic Press.}
\bibitem{Ke} {Keilson, J. (1979). Markov chain models-rarity and exponentiality. New York: Springer-Verlag.}
\bibitem{Ku} {Kurtz, J., (1970). Solutions of ordinary differential equations as limits of pure jump Markov processes. J. Appl. Prob., 7: 49-58. }
\bibitem{AL} {Lambert, A., (2005). The branching process with logistic growth. Ann. Appl. Prob., 15: 1506-1535.}
\bibitem{Su}{Sumita,U., (1984). On conditional passage time structure of birth-death processes. J. Appl. Prob., 21(1): 10-21.}

\end{thebibliography}
\end{document}